\pdfoutput=1
\immediate\write18{makeindex \jobname.nlo -s nomencl.ist -o \jobname.nls}
\documentclass[journal]{IEEEtran}
\usepackage{amsmath,amsfonts,amssymb,amsthm}

\newtheorem{theorem}{Theorem}[section]
\newtheorem{lemma}[theorem]{Lemma}
\newtheorem{property}[theorem]{Property}
\newtheorem{definition}[theorem]{Definition}
\newtheorem{corollary}[theorem]{Corollary}

\usepackage{graphicx}
\usepackage{float}
\usepackage{caption}
\usepackage{subcaption}
\usepackage{bm} 
\usepackage{dsfont}
\usepackage{wasysym}
\usepackage{nameref}
\usepackage{enumerate}
\usepackage{nomencl}
\usepackage{cite}
\usepackage{xcolor}
\usepackage{soul}

\usepackage[pdftex,breaklinks,hidelinks]{hyperref}

\makenomenclature
\nomenclature[03]{$\mathcal{N}$}{Set of all devices}
\nomenclature[17]{$\mathcal{N}_j$}{$j$th subset of devices}
\nomenclature[15]{$\mathcal{X}$}{State-space}
\nomenclature[21]{$\mathcal{F}_{\bar{p},x}$}{Feasible set from state $x$, with vector of maximum powers $\bar{p}$}
\nomenclature[20]{$P^r(t)$}{Reference at time $t$}
\nomenclature[21]{$P^r_{[t_0,t)}$}{Reference signal truncated to the interval $[t_0,t)$}
\nomenclature[02]{$D_i$}{$i$th device}
\nomenclature[13]{$x_i(t),\ z_i(t)$}{Time-to-go of the $i$th device at time $t$}
\nomenclature[14]{$x(t),\ z(t)$}{Vector of time-to-go values at time $t$, across all devices}
\nomenclature[18]{$x_{\mathcal{N}_j}(t)$}{Equal time-to-go value of all members of $\mathcal{N}_j$, at time $t$}
\nomenclature[16]{$X(t)$}{Vector of distinct time-to-go values at time $t$}
\nomenclature[05]{$u_i(t)$}{Power output of the $i$th device at time $t$}
\nomenclature[06]{$u(t)$}{Vector of power outputs at time $t$, across all devices}
\nomenclature[19]{$\bar{U}_j$}{Vector of maximum power outputs, across the $j$th subset of devices}
\nomenclature[08]{$\bar{p}_i$}{Maximum power rating of the $i$th device}
\nomenclature[09]{$\bar{p}$}{Vector of maximum power ratings, across all devices}
\nomenclature[10]{$P$}{Diagonal matrix of maximum power ratings, across all devices}
\nomenclature[04]{$e_i(t)$}{Energy of the $i$th device at time $t$}
\nomenclature[01]{$n$}{Total number of devices}
\nomenclature[24]{$R(t)$}{Worst-case reference at time $t$}
\nomenclature[23]{$E_{P^r}(p)$}{$E$-$p$ transform of reference $P^r(\cdot)$, evaluated at power level $p$}
\nomenclature[25]{$\Omega_{\bar{p},x}(p)$}{Capacity curve from state $x$, with vector of maximum powers $\bar{p}$, evaluated at power level $p$}
\nomenclature[12]{$\mathcal{U}_{\bar{p}}$}{Product set of power constraints, with vector of maximum powers $\bar{p}$}

\begin{document}
	\renewcommand{\qed}{\null\hfill\IEEEQEDclosed}
 	\bstctlcite{MyBSTcontrol}
	\title{A Graphical Measure of Aggregate Flexibility\\ for Energy-Constrained Distributed Resources}%
	\author{Michael P. Evans, Simon H. Tindemans, \IEEEmembership{Member, IEEE} and David Angeli, \IEEEmembership{Fellow, IEEE}
	\thanks{M.P. Evans and D. Angeli are with the Department of Electrical and Electronic Engineering, Imperial College London, UK (email: m.evans16@imperial.ac.uk; d.angeli@imperial.ac.uk). D. Angeli is in addition with the Department of Information Engineering, University of Florence, Italy. S.H. Tindemans is with the Department of Electrical Sustainable Energy, TU Delft, Netherlands (email: s.h.tindemans@tudelft.nl). This work was supported by EPSRC studentship 1688672.}}
	\date{}
	
	\markboth{Submitted to Transactions on Smart Grid}{}
	\maketitle
		
	\begin{abstract}
		We consider the problem of dispatching a fleet of heterogeneous energy storage units to provide grid support. Under the restriction that recharging is not possible during the time frame of interest, we develop an aggregate measure of fleet flexibility with an intuitive graphical interpretation. This analytical expression summarises the full set of demand traces that the fleet can satisfy, and can be used for immediate and straightforward determination of the feasibility of any service request. This representation therefore facilitates a wide range of capability assessments, such as flexibility comparisons between fleets or the determination of a fleet's ability to deliver ancillary services. Examples are shown of applications to fleet flexibility comparisons, signal feasibility assessment and the optimisation of ancillary service provision. 
	\end{abstract}

	
	\printnomenclature[1.8cm]
	
	\section{Introduction}
	\label{sec:intro}
	\subsection{Background}
	We consider the problem of operating a fleet of storage devices, i.e. a collection of these which are dispatched in a coordinated manner, in order to provide system support in supply-shortfall conditions. Recent years have seen the proliferation of multiple forms of storage onto electricity grids and, crucially, many of these resources are energy-constrained, for example due to physical capacity limits or operational limits set by users. These storage technologies can provide a range of valuable services to the system, with one advantage being their ability to shift consumption in time, thereby compensating for fluctuations in the output of intermittent generation. This offers significant potential as a means to replace conventional generation as electricity networks are decarbonised, according to the World Energy Council \cite{WEC}. Increasingly, system operators are offering frameworks for convenient delivery of grid-support services; to date these include the California Independent System Operator (CAISO) \cite{CAISO} and the Pennsylvania-New Jersey-Maryland Interconnection RTO (PJM) \cite{PJM} in the US. The latter offers the control framework within which we base our research: a system-wide regulation signal, representing the mismatch between electricity supply and demand, is broadcast and service participants endeavour to track this signal. Thus distributed resources are controlled centrally via their aggregate response. For the purposes of description we assume that an aggregator, defined as the entity responsible for the provision of system services through dispatch of the fleet, is contracted by the network operator as follows. The amount of power requested is updated at regular intervals, and the aggregator must make a decision as to which of its resources to deploy to meet the request. This aggregation of small participants into composite entities occurs in many ancillary service markets with minimum size requirements. Example fleets include uninterruptible power supplies with storage headroom and electric vehicle (EV) batteries returned to the manufacturer for recycling; grid support offers a convenient end-of-life deployment for the latter once they have deteriorated beyond usability in the EV. Note that if significant loss of load were at stake the network operator might instead take on the role of central dispatcher.
	
	We focus on the decision making of the aggregator, and consider how to best dispatch devices in the presence of uncertain demand. Previous literature has addressed the coordinated dispatch of a fleet of distributed devices in comparable problem frameworks. Prior approaches can be broadly clustered into three categories: optimal control, mean-field control \cite{Tindemans2015,Busic2016,Chertkov2018} and transactive energy \cite{Chen2017,Nunna2017,Hu2017,Liu2018}. Our approach lies in the first category and we directly optimise power flows between storage and the grid via explicit feedback policies. In contrast to mean field approaches, this allows us to achieve guaranteed optimal behaviour, as opposed to in expectation. 
	
	Within applications of optimal control to the dispatch of distributed devices, a common approach has been the composition of an optimisation problem of standard form for the entire network and across the full time horizon. The dispatch of each device at each time instant is then assigned as a decision variable in the optimisation problem. Objectives chosen have included minimisation of charging costs \cite{Alharbi2017} or more generalised operational costs \cite{Alharbi2017,Paterakis2016,Alharbi2015,Battistelli2012,Bai2015,Saber2012}, or flattening of load profiles \cite{Alonso2014}. Varying based upon the problem framework, solution techniques applied to these problems have included linear solvers \cite{Paterakis2016}, mixed-integer linear solvers \cite{Alharbi2015}, combined mixed-integer linear/nonlinear solvers \cite{Alharbi2017}, robust optimisation \cite{Battistelli2012,Bai2015}, particle swarm algorithms \cite{Saber2012} and genetic algorithms \cite{Alonso2014}. 
	
	We are interested in achieving optimality in a simpler way, negating the need to perform the large computations of this prior work. Moreover, we are interested in applying our techniques to settings with a complete lack of knowledge about future request signals, including probability distributions or forecasts. In these cases one would be unable to compose an optimisation problem covering all future time instants. For example, \cite{Gan2013} composed a simpler optimisation problem and presented a decentralised solution, but still required a foreknown demand profile.
	
	Specifically, we consider approaches where the optimal dispatch problem is divided into two coupled sub-problems. A real-time \textit{control} algorithm dispatches the fleet of devices according to a common control signal, without knowledge of the future (i.e. a greedy algorithm). This control algorithm is paired with a matching \textit{scheduling} problem that determines the capability of the fleet to meet future requests. It uses an aggregate representation of the flexibility limits of the device fleet to ensure that only feasible responses are scheduled. An example of this approach, applied to mean field control of thermostatically controlled loads, can be found in \cite{Tindemans2015} and \cite{Trovato2015}, for the control and scheduling problems respectively.
	
	A number of previous proposals to control single storage devices can be interpreted through the lens of a greedy control strategy plus a separate scheduling component. \cite{Bejan2012,Gast2014,Cruise2014} all implemented a greedy buffer policy in which the device charges as quickly as possible under conditions of excess supply and discharges as quickly as possible under excess demand, but chose different reference levels for defining this mismatch. In \cite{Bejan2012}, thermal plant was dispatched so that it was forecast to maintain the storage unit at a predefined set level, and buffering was done in real-time to counter forecast errors. In \cite{Gast2014} the buffer comparison level was instead the forecast net demand plus either a fixed offset or an offset based upon the forecast storage level. Cruise~\emph{et al.}\cite{Cruise2014} were concerned with arbitrage applications, where the authors showed that the choice to maximally charge or discharge should be based upon the notional value per unit of energy stored as follows: charge if this is above the buy price, discharge if this is below the sell price, otherwise do nothing. 
	
	Results with strict optimality proofs can be obtained by restricting the operational regime to discharging only. This corresponds to an assumption that the ability to satisfy the power request takes precedence over other objectives, allowing one to study properties of the system in a general sense, without considering price dynamics in detail. A natural objective for the aggregator within this setting is to postpone the time to failure, i.e. the time at which the aggregator is first unable to meet the requested power supply. In general, this approach is applicable whenever an event occurs which threatens security of supply, because it would give the network operator the most time in which to attempt recovery of normal operating conditions. A straightforward example of such an event would be the loss of a generator or transformer, or an islanding fault. Other relevant events might occur within nominally ordinary conditions, for example an unexpected shortfall in wind output or peak in demand. In these cases, purely surviving until the natural termination of the event would be sufficient to recover normal operation. In \cite{Edwards2017}, the authors showed that, for a single energy storage device, the naive greedy discharge policy maximises inclusion in the set of future request signals that the device is able to meet. This policy therefore maximises time to failure. The corresponding scheduling problem is trivial for a single device, because the feasibility of a specific output signal is immediately determined from the power and energy limits of the device.
	
	In \cite{Evans2017}, we proposed a greedy control policy that is optimal in this same sense for a fleet of arbitrary size. A feedback policy (which we will here refer to as the \textit{optimal} policy) was presented and shown to maximise the time to failure, instantaneous maximum power and flexibility of a system supplied by energy-constrained distributed resources. The flexibility maximisation, in particular, was performed in a set-theoretic sense; which is to say that the set of future feasible scenarios is always as large as possible, regardless of the current fleet output. This can be interpreted to mean that, with no information about future signals, implementation of the policy is guaranteed to result in the latest failure time of the system. Moreover, by forming a general integrator model of storage, we are able to apply our analysis across a wide range of device types that have previously been considered within the paradigm of demand response \cite{Callaway2010}. These include EVs \cite{Shao2012,Wenzel2017}, home storage devices \cite{Tsui2012,Wang2013} and diesel generators \cite{Mazidi2014,Wang2015}. It should be noted that a comparable policy to that which we propose, applied to a continuum of devices, is also relevant in price-determined settings \cite{DePaola2017} and enables the coordination of devices within areas of flat prices.

	\subsection{Contribution of this paper}
	We argue that the policy of \cite{Evans2017} solves, in an optimal sense, the \emph{control problem} for a fleet of energy-constrained distributed resources. While this policy can be used to determine fleet capability in a procedural way, i.e. by running the algorithm, this paper proposes an analytical method for undertaking the same assessment We present a transform that returns an aggregate representation of the fleet; this therefore enables longer-timescale planning and so can be used to solve the associated \emph{scheduling problem}. Moreover, this transform can also be applied to a received request for immediate feasibility determination, which leads to greater insight into fleet capabilities. The contributions of this paper can be summarised as follows:
	
	\begin{itemize}
		\item We present a functional means to determine request feasibility, thereby improving upon the procedural application of the optimal policy of \cite{Evans2017}.
		\item We explore the range of capability assessments that the transform enables: specification of maximum system service provision; immediate feasibility determination for a received request and flexibility comparison between different fleets.
		\item We discuss how this transform has multiple useful properties that enable intuitive fleet characterisation.
		\item We show how this transform can be implemented graphically, thereby enabling feasibility determination by eye. 
		\item We show how a graphical comparison can be made to the homogeneous fleet of the same total ratings.
		\item We derive additional analytical results relating to feasibility of requests, which lead to a greater understanding of the problem in general.
	\end{itemize}
	
	In deriving the transform, we utilise the result of \cite{Evans2017} that there is a one-to-one mapping between feasible signals and those that will be met by the optimal policy. Use of the developed framework can be seen in Figure~\ref{fig:flow1}. 
			
	\begin{figure}[h]
		\includegraphics[width=\columnwidth]{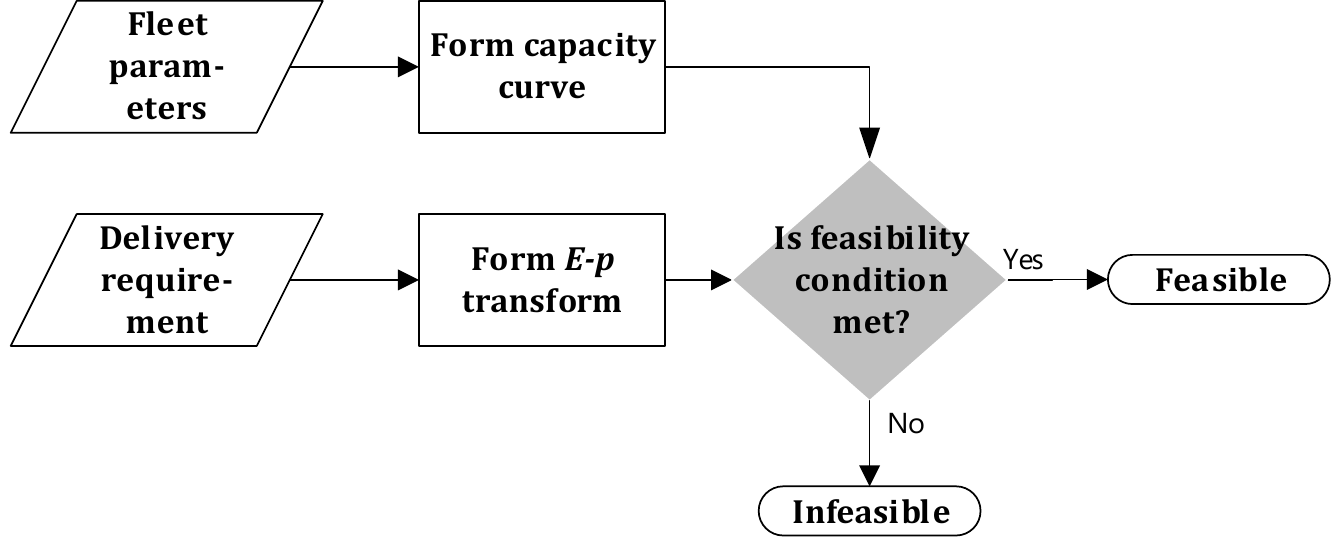}
		\caption{Use of the presented framework to determine the feasibility of a given delivery requirement.}
		\label{fig:flow1}
	\end{figure}
	
	\vspace{-3mm}
	\subsection{Motivating example}
	\label{sec:motivating example}
	The technique that we present is able to compare fleets of devices utilised for system services, when often such a comparison appears far from trivial. To this end, the contribution of this paper is perhaps best motivated by means of an example. Consider the microgrid of Figure~\ref{fig:microgrid}, onto which the system operator plans to connect a storage setup with one of the following three device configurations (energy, power): A, consisting of (108~kWh, 4~kW) and (36~kWh, 18~kW); B, consisting of (104~kWh, 13~kW) and C, consisting of (90~kWh, 8~kW) and (54~kWh, 14~kW). Note that configurations A and C have the same total energy and power but distributed differently between devices. Configuration B has smaller values of each but is comprised of a single device. Hence, prior to the analysis that we present, it is unclear which configuration would be the best choice.
	
	\begin{figure}[h]
		\centering
		\includegraphics[width=0.4\columnwidth]{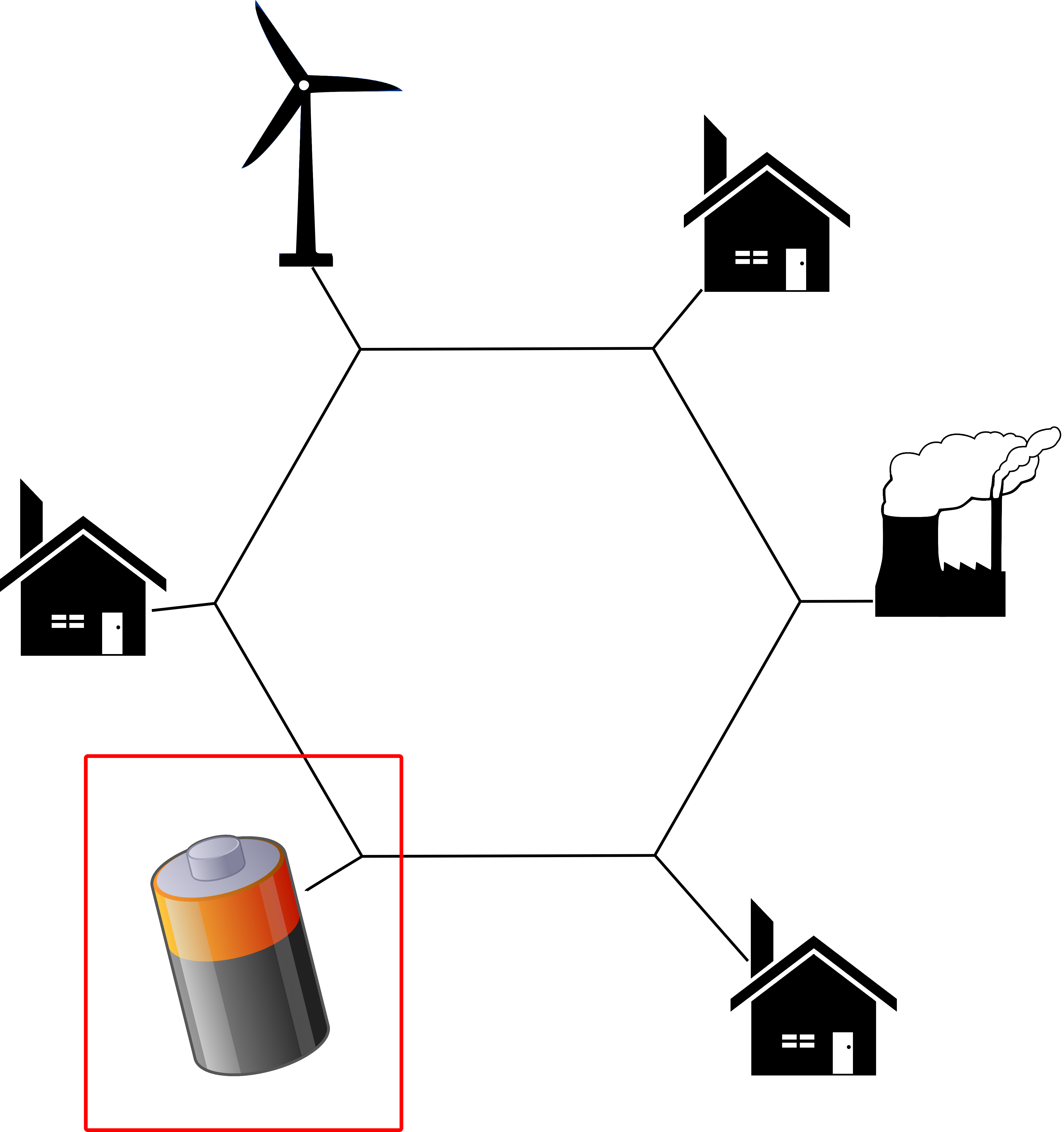}
		\caption{An example microgrid, onto which the network operator plans to connect storage at the location highlighted by the red box.}
		\label{fig:microgrid}
	\end{figure}
	
	\vspace{-3mm}
	\subsection{Organisation of the paper}
	The remainder of this paper is organised as follows. Section~\ref{sec:maths formulation} covers the mathematical formulation of the problem. This is followed by a presentation of analytical results in Section~\ref{sec:additional results}, and the application of these results to numerical examples in Section~\ref{sec:numerical results}. Section~\ref{sec:disc} then discusses the implications of these results. Finally, Section~\ref{sec: conclusion} concludes the report and discusses relevant future work. 
	
	\section{Mathematical Formulation}
	\label{sec:maths formulation}
	\subsection{Problem description}
	We denote by $n$ the number of energy-constrained storage units available to the aggregator, and define the set of all devices as $\mathcal{N}\doteq\{D_1,D_2,...,D_n\}$. We do not impose any restrictions on homogeneity of devices and allow each device to have a unique discharging efficiency. For convenience we incorporate this into the model implicitly by considering the extractable energy of each device, $e_i(t)$. We choose the power delivered by each device to be the control input $u_i(t)$, and assume that this is measured externally so that efficiency is once again accounted for. This leads to integrator dynamics on the energy of each device,
	\begin{equation}
	\dot{e}_i(t)=-u_i(t),
	\end{equation}
	subject to the assumed physical constraint $e_i(t)\geq0$. We neglect network effects, in particular due to the high likelihood that storage devices are situated near to electricity-consuming devices, so that their deployment is unlikely to be constrained by network congestion or losses. We also assume the absence of cross-charging between devices, which corresponds to a regime of energy scarcity. We therefore restrict our devices to discharging operation only, so that the power of each device is constrained as $u_i(t)\in[0,\bar{p}_i]$, in which $\bar{p}_i$ denotes the maximum discharge rate of device $D_i$, and with the  convention that discharging rates are positive. Note, once again, that there is no homogeneity imposed on energy or power constraints across the ensemble of devices. We define the \textit{time-to-go} of device $D_i$ to be the remaining time for which it can run at its maximum power, i.e.
	\begin{equation}
	x_i(t) \doteq \frac{e_i(t)}{\bar{p}_i},
	\end{equation}	
	and represent the state of each device by its time-to-go. We then form state, input and maximum power vectors as
	\begin{alignat}{3}
	x(t)&\doteq \begin{bmatrix}
	x_1(t)\ \dots\ x_n(t)
	\end{bmatrix}^T,\\
	u(t) &\doteq \begin{bmatrix}
	u_1(t)\ \dots\ u_n(t)
	\end{bmatrix}^T,\\
	\bar{p}&\doteq\begin{bmatrix}
	\bar{p}_1\ \dots\	\bar{p}_n
	\end{bmatrix}^T
	\end{alignat}
	respectively, so that we can write our dynamics in matrix form as $\dot{x}(t)=-P^{-1}u(t)$, in which $P \doteq \text{diag}(\bar{p})$. We define the state space as $\mathcal{X}\doteq [0,+\infty)^n$, and form the product set of our constraints on all the inputs,
	\begin{equation}
	\mathcal{U}_{\bar{p}} \doteq [0,\bar{p}_1] \times [0,\bar{p}_2] \times ... \times [0,\bar{p}_n],
	\end{equation}	
	allowing us to write our input constraints as $u(t)\in\mathcal{U}_{\bar{p}}$.
	
	Now, as in \cite{Evans2017}, we partition the devices considered at a chosen time instant into collections of descending equal state value, without loss of generality leading to
	\begin{equation}
	x_1=...=x_{s_1} >... > x_{s_{q-1}+1}=...=x_{s_q},
	\end{equation}
	in which $q$ is the number of unique $x_i$-values. Note that we have dropped the explicit dependence on $t$ as we are considering a single time instant. We then denote the subsets formed as $\mathcal{N}_i\doteq\{D_j\colon x_j=x_{s_i}\},\ i=1,2,...,q$, and similarly denote by $x_{\mathcal{N}_i}\doteq x_{s_i},\ i=1,2,...,q$ the (single) time-to-go value of the devices in each subset. We finally form a condensed state vector of distinct values only as
	\begin{equation}
	X\doteq\begin{bmatrix}
	x_{\mathcal{N}_1}\ \dots\ x_{\mathcal{N}_{q}}
	\end{bmatrix}^T,
	\end{equation}
	and denote the maximum power vectors corresponding to each subset as
	\begin{equation}
	\bar{U}_i\doteq\begin{bmatrix}
	\bar{p}_{s_{i-1}+1}\ ...\ \bar{p}_{s_i}
	\end{bmatrix}^T,\ \ i=1,...,q,
	\end{equation}
	with the convention that $s_0=0$.
	
	We denote by $P^r\colon [0,+\infty) \mapsto[0,+\infty)$ a power reference signal received by the aggregator, and in addition denote a truncation of such a signal as
	\begin{equation}
	P^r_{[t_0,t)}\doteq\left\{
	\begin{array}{@{}ll@{}}
	P^r(\tau), & \text{if}\ \tau\in[t_0,t) \\
	0, & \text{otherwise}.
	\end{array}\right.
	\end{equation} 
	We utilise equivalent notation for the truncation of any other signal also. 
	For a reference to be \textit{feasible}, there must exist a control signal able to satisfy it for all time without violating any constraints. These are both the power constraints on the control signal itself and the energy constraints that apply to the resulting state trajectory.  We therefore define the set of feasible reference signals as follows:
	\begin{definition}
		\label{def:feasible}
		The set of feasible power reference signals, for a system with maximum power vector $\bar{p}$ and initial state $x=x(0)$, is defined as
		\begin{equation*}
		\begin{aligned}
		\mathcal{F}_{\bar{p},x} \doteq \big\{&P^r(\cdot) \colon \exists u(\cdot),\ z(\cdot)\colon \forall t\geq 0,\ 1^Tu(t)=P^r(t),\\
		&u(t)\in \mathcal{U}_{\bar{p}},\ \dot{z}(t)=-P^{-1}u(t),\ z(0)=x,\ z(t) \geq 0 \big\}.
		\end{aligned}
		\end{equation*}
	\end{definition}
	
	\noindent We interpret the inclusion of a given feasible set by another as an increase in \emph{flexibility}. 

		\subsection{Optimal Feedback Policy}
		Due to its maximisation of the feasible set, we will continually refer to the optimal policy of \cite{Evans2017}. We therefore reproduce it here as follows, in which it should be noted that explicit time dependencies have been dropped as we are again considering instantaneous values. Without loss of generality, form subsets of devices with equal value in descending order, of which there will be $q$, so that $x_{\mathcal{N}_1}>x_{\mathcal{N}_2}>...>x_{\mathcal{N}_q}$. The explicit feedback law is then calculated as a fraction $r_i$ of the maximum power $\bar{U}_i$ according to:
		\begin{subequations}
			\begin{gather}
			r_i=\left\{
			\begin{array}{@{}ll@{}}
			1, & \text{if}\ \sum_{j\leq i}1^T\bar{U}_j\leq P^r \\
			0, & \text{if}\ \sum_{j< i}1^T\bar{U}_j\geq P^r \\
			\frac{P^r-\sum_{j< i}1^T\bar{U}_j}{1^T\bar{U}_i}, & \text{otherwise},
			\end{array}\right.\\
			u^*(x,P^r)\doteq\begin{bmatrix}
			r_1\bar{U}_1^T\ \dots\ r_q\bar{U}_q^T
			\end{bmatrix}^T.
			\end{gather}
			\label{eq: explicit policy}
		\end{subequations}
		\hspace{-3mm} In words, this policy allocates devices in descending order of time-to-go, at maximum power (with up to one subset of devices run at a fraction of maximum power to exactly meet the reference). This then depletes as few devices as possible, thereby maximising the number of available devices and so giving the fleet the best chance to meet an unknown future reference.
		
		Denoting by $z^*(\cdot)$ the state trajectory under the application of \eqref{eq: explicit policy}, the closed-loop dynamics are
		\begin{equation}
		\label{eq:closed-loop dynamics}
		\dot{z}^*(t)=-P^{-1}u^*\big(z^*(t),P^r(t)\big) .
		\end{equation}
		
	\section{Results on Feasibility}
	\label{sec:additional results}
	In this section we present theoretical results which hold in general. These results are of greater importance to our arguments than their derivations and so, for clarity of argument, their proofs have been omitted from the main body of this paper and can be found in the \nameref{sec:appendix}.

	\subsection{Interesting system properties}
	The existence of a feasible set which includes all others (maintained via implementation of the optimal policy \eqref{eq: explicit policy} - see \cite{Evans2017} for more details) allows us to derive additional useful properties of the system considered. Moreover, the optimality of the presented policy allows us to utilise it as a proxy for this largest feasible set. We undertake such derivations here, implicitly utilising this policy in each case. The proofs of Lemmas~\ref{lem:commutable refs} and \ref{lem:terminal state} can be found in the \nameref{sec:appendix}.
	\begin{lemma}
		Piecewise constant reference signals are permutable whilst maintaining feasibility.
		\label{lem:commutable refs}
	\end{lemma}
	\newcommand{\proofCommutableRefs}{
		Given a feasible reference signal $P^r(\cdot)$, consider the class of modified signals defined by:
		\begin{enumerate}
			\item An arbitrary partition of $[t_1,T)$,
		\begin{equation}
		P^r(\cdot)=P^r_{[t_1,t_2)}\ \APLlog\ P^r_{[t_2,t_3)}\ \APLlog\ ...\ \APLlog\ P^r_{[t_{v},t_{v+1}=T)},
		\end{equation}
		in which $v$ is the number of parts and the $\APLlog$ operator denotes the concatenation of two signals.
		
		\item An arbitrary permutation of the part indices $1,2,...,v$. Letting $\psi$ denote the vector dictating the order of these parts, the partition induced by $\psi$ is then
		\begin{equation}
		\tilde{P}^r(\cdot)\doteq P^r_{[t_{\psi_1},t_{\psi_1+1})}\ \APLlog\ P^r_{[t_{\psi_2},t_{\psi_2+1})}\ \APLlog\ ...\ \APLlog\ P^r_{[t_{\psi_v},t_{\psi_v+1})}.
		\end{equation}
		\end{enumerate} 
	
		\noindent $P^r(\cdot)$ is feasible, and so let $x(\cdot)$ and $u(\cdot)$ satisfy the following conditions for all non-negative $t$: 
			\begin{equation}
			\begin{aligned}
			&u(t)\in \mathcal{U}_{\bar{p}},\ 1^Tu(t)=P^r(t),\\
			&\dot{x}(t)=-P^{-1}u(t),\  x_i(t) \geq 0,\ \ i=1,2,...,n.
			\end{aligned}
			\end{equation}		
		Moreover, due to the inability of devices to charge, we can utilise the fact that $x(t)\geq0\ \ \forall t\iff x(T)\geq0$ for an expression equivalent to the last condition. Now, consider permuting the input $u(\cdot)$ in time, corresponding to the permutation of the reference, to produce $\tilde{u}(\cdot)$, i.e.
		\begin{equation}
		\tilde{u}(\cdot)\doteq u_{[t_{\psi_1},t_{\psi_1+1})}\ \APLlog\ u_{[t_{\psi_2},t_{\psi_2+1})}\ \APLlog\ ...\ \APLlog\ u_{[t_{\psi_v},t_{\psi_v+1})},
		\end{equation}
		in which $u_{[t_1,t_2)}$ denotes the signal $u(\cdot)$ truncated to the half-open interval $[t_1,t_2)$. Consider each condition of the feasibility of $\tilde{P}^r(\cdot)$ individually as follows. By construction, for all non-negative $t$,
		\begin{equation}
		\begin{aligned}
		\tilde{u}(t)&\in\mathcal{U}_{\bar{p}}, \\
		1^T\tilde{u}(t)&=\tilde{P}^r(t), \\
		\dot{\tilde{x}}(t)&=-P^{-1}\tilde{u}(t),
		\end{aligned}
		\end{equation}
		in which $\tilde{x}(\cdot)$ denotes the trajectory that fulfils the permuted reference and is initialised as $\tilde{x}(0)=x(0)$. In addition,
		\begin{equation}
		\begin{aligned}
		\tilde{x}(T)&=\tilde{x}(0)-\int_0^TP^{-1}\tilde{u}(\tau)d\tau\\
		&=x(0)-\int_0^TP^{-1}{u}(\tau)d\tau={x}(T)\geq0.
		\end{aligned}
		\end{equation}
		Therefore $\tilde{P}^r(\cdot)$ satisfies all the necessary conditions and we are able to deduce that $P^r(\cdot)\in\mathcal{F}_{\bar{p},x} \iff 	\tilde{P}^r(\cdot)\in\mathcal{F}_{\bar{p},x}$,
		in which the condition is necessary as well as sufficient as a result of its symmetry.
		\qed
	}
	
	\noindent One can interpret this result to mean that the ability of a system to satisfy a given reference signal is independent of the variation of that signal in time; rather it is dependent on the total time spent at each power level.

	\begin{lemma}
		\label{lem:terminal state}
		Given two feasible piecewise constant reference signals which are permuted versions of one another, the final state under both references is the same when the optimal policy is used.
	\end{lemma}
	\newcommand{\proofTerminalState}{
	Let the system under consideration have initial state $x$ and maximum power vector $\bar{p}$. Then consider partitioning a feasible piecewise constant reference signal $P^r(\cdot)$, of length $T$, into the partitions before and after time $t$, i.e. $P^r(\cdot)=P^r_{[0,t)}\ \APLlog\ P^r_{[t,T)}$. Now form a new signal as $\hat{P}^r(\cdot)\doteq\tilde{P}^r_{[0,t)}\ \APLlog\ P^r_{[t,T)}$, in which $\tilde{P}^r_{[0,t)}$ is a permuted version of $P^r_{[0,t)}$. Using the result of Lemma~\ref{lem:commutable refs}, we know that if $P^r(\cdot)$ is feasible then so too must be $\hat{P}^r(\cdot)$. 
	
	Denote by $z$ and $\tilde{z}$ the states reached at time $t$ under the application of the optimal policy to meet $P^r(\cdot)$ and $\hat{P}^r(\cdot)$ respectively. Without loss of generality, let the vector $x$ be ordered in decreasing time-to-go, and let the same device ordering be used for $z$ and $\tilde{z}$. Then, because ordering is preserved under the optimal policy, $z$ and $\tilde{z}$ will also be in descending order. From the above result, we are able to say that $P^r_{[t,T)}\in\mathcal{F}_{\bar{p},z}$ and $P^r_{[t,T)}\in\mathcal{F}_{\bar{p},\hat{z}}$. As this holds for any feasible choice of reference $P^r(\cdot)$, we are then able to say that $\mathcal{F}_{\bar{p},\hat{z}}=\mathcal{F}_{\bar{p},z}$. We use this to prove the required result by contradiction as follows. Consider firstly the case in which $z_n>\tilde{z}_n$, in which $n$ indexes the device with the smallest time-to-go value. We are able to construct the following reference signal:
	\begin{equation}
	P^r(t)=\left\{
	\begin{array}{@{}ll@{}}
	\sum_{i=1}^n\bar{p}_i, & \text{if}\ t\in[0,z_n) \\
	0, & \text{otherwise},
	\end{array}\right.
	\end{equation}
	for which we can trivially see that $P^r(\cdot)\in\mathcal{F}_{\bar{p},z}$ and $P^r(\cdot)\notin\mathcal{F}_{\bar{p},\tilde{z}}$, hence the feasible sets defined by the two states must be distinct. 
	
	Now, consider an arbitrary device index $j$ and assume that 
	\begin{subequations}
		\begin{gather}
		z_i=\tilde{z}_i,\ i=n,n-1,...,j+1,\\
		z_j> \tilde{z}_j.
		\end{gather}
	\end{subequations}
	We are able to construct the following reference signal:
	\begin{equation}
	P^r(t)=\left\{
	\begin{array}{@{}ll@{}}
	\sum_{i=1}^n\bar{p}_i, & \text{if}\ t\in[0,z_n) \\
	\sum_{i=1}^{n-1}\bar{p}_i, & \text{if}\ t\in[z_n,z_{n-1})\\
	\vdots \\
	\sum_{i=1}^{j}\bar{p}_i, & \text{if}\ t\in[z_{j+1},z_j)\\
	0, & \text{otherwise},
	\end{array}\right.
	\end{equation}
	for which we can trivially see that $P^r(\cdot)\in\mathcal{F}_{\bar{p},z}$ and $P^r(\cdot)\notin\mathcal{F}_{\bar{p},\tilde{z}}$, hence the feasible sets defined by the two states must be distinct. Induction of this argument from device index $n$ to $1$ returns the result that the feasible sets must be distinct whenever $\exists k\colon z_k>\tilde{z}_k$. In addition, due to the arbitrary allocation of $z$ and $\tilde{z}$, this result must hold whenever $z\neq\tilde{z}$. Hence the permutation must not alter the terminal state.
	\qed
}

\subsection{$E$-$p$ transform}
The ability to utilise the policy \eqref{eq: explicit policy} as a proxy for the feasible set also enables the derivation of a graphical representation of this set. In this section we present a transform, followed by a discussion of its uses to a grid operator or aggregator. The proofs of Lemma~\ref{lem:convex} and Theorem~\ref{the:E_p} can be found in the \nameref{sec:appendix}.
\begin{definition}
	Given a power reference $P^r\colon [0,+\infty)\mapsto[0,+\infty)$, we define its E-p transform as the following function:
	\begin{equation}
	E_{P^r}(p)\doteq\int_0^\infty \textnormal{max}\big\{P^r(t)-p,0\big\}dt,
	\end{equation}
	interpretable as the energy supplied above any given power rating, $p$.
\end{definition}
\noindent The following properties of this transform are of particular relevance to our analysis:
\begin{property}
	The $E$-$p$ transform intersects the $E-$ and $p-$axes at $\int_0^\infty P^r(t)dt$ (total energy supplied) and $\underset{t}{\textnormal{sup}}\ P^r(t)$ (maximum power supplied) respectively.
\end{property}
\begin{lemma}
	\label{lem:convex}
	The $E$-$p$ transform is convex and monotone.
\end{lemma}
\newcommand{\proofConvex}{
	Consider two arbitrary power levels, $p$ and $q$ for which $p\geq q$. These must satisfy
	\begin{equation}
	\text{max}\big\{P^r(\tau)-p,0\}\leq \text{max}\big\{P^r(\tau)-q,0\}\ \ \forall \tau,
	\end{equation}
	and so $E_{P^r}(p)\leq E_{P^r}(q)$, i.e. the $E$-$p$ curve is monotone.	In addition, the definition of the $E$-$p$ transform leads to the following left- and right-derivatives respectively:
	\begin{subequations}
	\begin{gather}
	\frac{dE(p)}{dp}^-=-\mu\big(\big\{\tau\colon P^r(\tau)\geq p\big\}\big),\\
	\frac{dE(p)}{dp}^+=-\mu\big(\big\{\tau\colon P^r(\tau)> p\big\}\big),
	\end{gather}
	\end{subequations}
	in which $\mu(\cdot)$ denotes the Lebesgue measure operator.	Hence
	\begin{equation}
	0\geq\frac{dE(p)}{dp}\bigg\vert_p\geq\frac{dE(p)}{dp}\bigg\vert_q,
	\end{equation}
	thus the gradient is negative and monotonically increasing, and so the curve must be convex.
	\qed
}
\noindent We additionally define the \textit{worst-case} reference signal, and form its $E$-$p$ transform as follows:

\begin{definition}
	\label{fact:capacity}
	The worst-case reference signal, $R(\cdot)$, that can be fulfilled by a given system runs all devices at full power until they deplete; which will take place in order of ascending time-to-go. 
	This signal can be calculated as follows:
	\begin{equation}
	R(t)\doteq\sum_{i=1}^n\bar{p}_i[H(t)-H(t-x_i)],
	\label{eq:worst-case_ref}
	\end{equation}
	in which $H(\cdot)$ denotes the Heaviside step function.
\end{definition}
\begin{definition}
	\label{definition: capacity}
	We define the \textbf{capacity} of a system, $\Omega_{\bar{p},x}(p)$, to be the $E$-$p$ transform of the worst-case reference signal that can be met by the system,
	\begin{equation}
	\Omega_{\bar{p},x}(p)\doteq E_{R}(p),
	\end{equation}
	in which $R(\cdot)$ is defined as in \eqref{eq:worst-case_ref}.
\end{definition}

\noindent We now consider comparing the $E$-$p$ transform of an arbitrary reference to the capacity as follows.
\begin{theorem}
	\label{the:E_p}
	A reference signal $P^r(\cdot)$ is feasible if, and only if, its $E$-$p$ transform is dominated by the capacity of the system, i.e. $E_{P^r}(p)\leq \Omega_{\bar{p},x}(p)\ \ \forall p\iff P^r(\cdot)\in\mathcal{F}_{\bar{p},x}$.
\end{theorem}
	
\noindent This result justifies our use of the terminology \emph{worst-case} reference signal for $R(\cdot)$ and \emph{capacity} for its $E$-$p$ transform. Definitions~\ref{fact:capacity} and \ref{definition: capacity} allow us to transform the complex object that is the set of feasible reference signals into a simple one-dimensional curve: that corresponding to the capacity of the system. This then enables a wide range of fleet capability assessments to be undertaken. For a given fleet, one could ascertain, for example, the longest ramp input of set gradient or highest-gradient ramp of a fixed duration that could be satisfied. Likewise, this could apply to pulse inputs, as we will demonstrate in Section~\ref{sec:numerical results}, or system services such as primary or secondary response.

If, in addition to the capacity curve, we form the $E$-$p$ transform of a received reference signal, Theorem~\ref{the:E_p} then allows for immediate and straightforward determination of the reference feasibility; simply by testing whether the capacity curve dominates the reference curve. This operation might provide valuable insight to a grid operator or aggregator when contemplating the feasibility of a request profile. Moreover, since the state fully defines the capacity of the system, the grid operator or aggregator would be able to make such deductions based solely upon the current state of their system. 

We are also able to use Theorem~\ref{the:E_p} to compare different device fleets as follows:
\begin{corollary}
	\label{cor:comparison}
	The feasible set of system-state pair $(\bar{p}^a,x^a)$ includes that of $(\bar{p}^b,x^b)$ if, and only if, the capacity curve of $(\bar{p}^a,x^a)$ dominates that of $(\bar{p}^b,x^b)$, i.e.
	\begin{equation}
	\Omega_{\bar{p}^a,x^a}(p)\geq\Omega_{\bar{p}^b,x^b}(p)\ \ \forall p\iff\mathcal{F}_{\bar{p}^a,x^a}\supseteq\mathcal{F}_{\bar{p}^b,x^b}.
	\end{equation}
\end{corollary} 

\noindent An alternative use of the $E$-$p$ transform would therefore be the comparison among fleets exemplified by the motivating example of the \nameref{sec:intro}. We also point out here that, as a result of the convexity and monotonicity of the transform, forming a capacity curve based upon lower bounds to the charge level of each device would allow any of these assessments to be implemented in a robust manner. Coupled with Lemma~\ref{lem:convex}, Corollary~\ref{cor:comparison} also allows us to deduce the following:

\begin{property}
	Given total energy and total power ratings $E_{max}$ and $p_{max}$ respectively, the most flexible distribution of devices has a capacity curve that is a straight line between $(0,E_{max})$ and $(p_{max},0)$. This is equivalent to a single device of the same ratings.
\end{property}
\begin{property}
	The deviation of a system's capacity curve from the maximum-flexibility curve of the same total ratings represents a \textbf{flexibility gap} resulting from its heterogeneous nature.
\end{property}

\vspace{-3mm}
\section{Numerical Results}
\label{sec:numerical results}
\subsection{Comparison between fleets}
We here present the capacities of the three device configurations discussed as a motivating example in the \nameref{sec:intro}. The three capacity curves can be seen in Figure~\ref{fig:2_battery_comp}, in which it can be seen that C is the unambiguous best choice as its $E$-$p$ curve dominates the others. Clearly, the more even distribution of C as compared to A offers more flexibility. When comparing A and B, if the network operator expects high-power, short-duration or low-power, long-duration request signals then A is preferred. If, however, signals with intermediate values of power and energy are expected, then B is preferred instead.
\begin{figure}[h]
	\centering
	\includegraphics{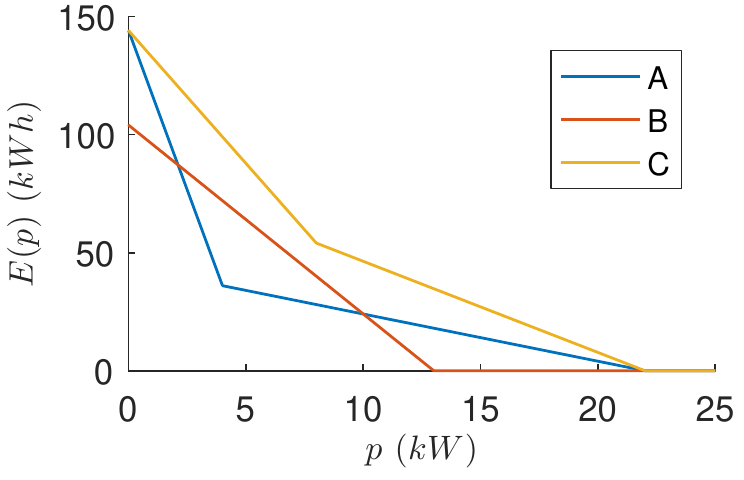}
	\caption{Capacity curves for the motivating example of Section~\ref{sec:motivating example}}
	\label{fig:2_battery_comp}
\end{figure}

\subsection{Feasibility of a received request}
\subsubsection{Alternative policy choices}
As in \cite{Evans2017}, we now compare indicative behaviour under the optimal policy to the implementation of two greedy heuristic alternatives. Note that \cite{Edwards2017} discusses sub-optimal policies such as these in more detail. The comparative policies chosen are as follows:

\begin{enumerate}[i)]
	\item \textit{Lowest Power First}. Order the devices by maximum power, without loss of generality leading to $\bar{p}_1\leq\bar{p}_2\leq...\leq\bar{p}_n$, and allocate devices in order from $D_1$ to $D_n$ as
\begin{equation}
u_i=\mathds{1}[x_i>0]\cdot\min\big\{\bar{p}_i,\ P^r-\sum_{j<i}u_j \big\},
\end{equation}	
in which $\mathds{1}[\cdot]$ denotes the indicator function. Note that the choice of allocation between devices of equal maximum powers is made arbitrarily. \\

\item \textit{Proportion of Power}. In this case, no ordering is required and each device is run according to
\begin{equation}
u_i=\mathds{1}[x_i>0]\cdot\min\bigg\{\bar{p}_i,\ \frac{\bar{p}_iP^r}{\sum_{i\colon x_i>0}\bar{p}_i}\bigg\}.
\end{equation}
\end{enumerate}
\subsubsection{Simulation results}
We compose a scenario in which there are 10,000 devices with initial time-to-go and maximum power values generated from uniform distributions, as $x_i\ {\sim}\ U(0,10)$~h and $\bar{p}_i\ {\sim}\ U(0,1.5)$~kW respectively. In an attempt to model a realistic dispatch, we choose a stepwise reference signal that is updated hourly, and draw each value $P^r[k]$ from the normal distribution $P^r[k]\ {\sim}\ N(2,0.8)$~MW. In this way all devices will be depleted by the end of a single day according to the optimal policy, and we set our simulation horizon to 1 day. We point out here that distributions are used solely for setting parameter values and therefore that this is a deterministic setup. The results of this case study can be seen in Figure~\ref{fig:AvailPlusRef}, and can be interpreted as follows. The optimal policy provides the highest feasible reference up to its time to failure, as it postpones emptying devices until absolutely necessary, resulting in the latest time to failure. If this scenario represents some failure mode requiring less than 12.8 h for resolution, any of the three policies are capable of maintaining full functionality. If, however, resolution requires between 13.2 and 16.3 h, the policy that we present is the only one out of the three that is capable of avoiding lost load. Beyond 16.3~h we are able to say that there exists no policy capable of avoiding lost load.

\captionsetup[figure]{font=small, skip=2pt}
\begin{figure}[h]
	\centering
	\includegraphics{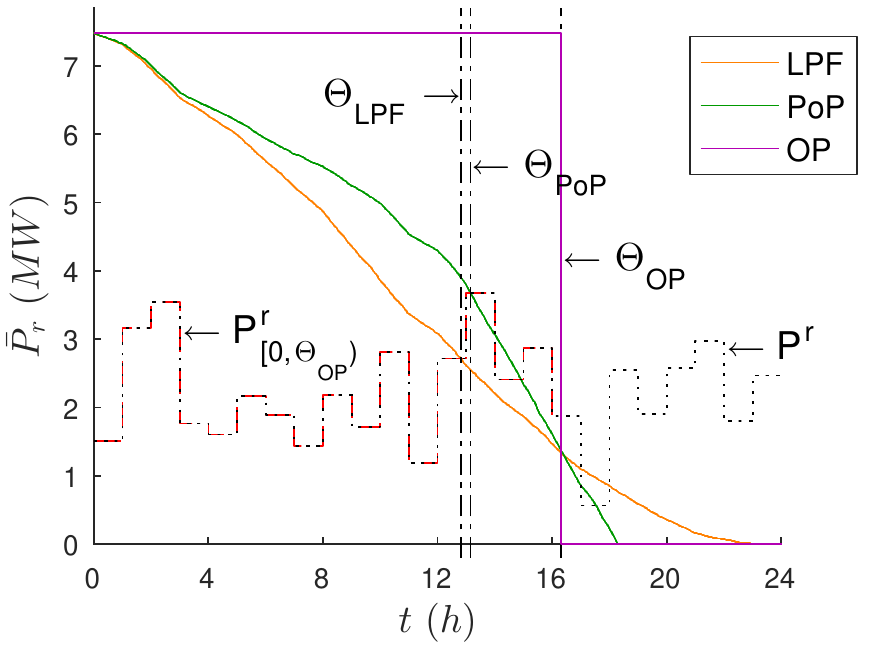}
	\caption{The maximum available power, $\bar{P}^r(t)$, under the implementation of the optimal (OP), lowest power first (LPF) and proportion of power (PoP) policies. The time to failure under each policy is represented by a black dash-dotted line and denoted $\Theta$. The corresponding reference signal is shown for comparison, both in full (the dotted black line) and truncated to the time to failure under the optimal policy (the dashed red line).}
	\label{fig:AvailPlusRef}
\end{figure}

\subsubsection{E-p analysis}
For the scenario described above, we here demonstrate the use of the $E$-$p$ transform in determining the feasibility of the full reference over the 24 h period (known to be infeasible) and in addition the same signal truncated to the time to failure under the optimal policy (feasible by construction), as can be seen in Figure~\ref{fig:E(p)}. We also plot the maximum-flexibility curve for comparison, and highlight the flexibility gap. As it can be seen, the truncation of the reference signal corresponds to bringing its $E$-$p$ transform to just below the capacity curve, i.e. into the feasible region. In addition, the $E$-$p$ and capacity curves intersect at $p=0$, which corresponds to the use of all available energy and the resulting depletion of all devices by the time to failure. Moreover, in this case there would be no advantage to homogenising the fleet, as the same truncation of the reference would be required to achieve feasibility.

\begin{figure}[h]
	\centering
	\includegraphics{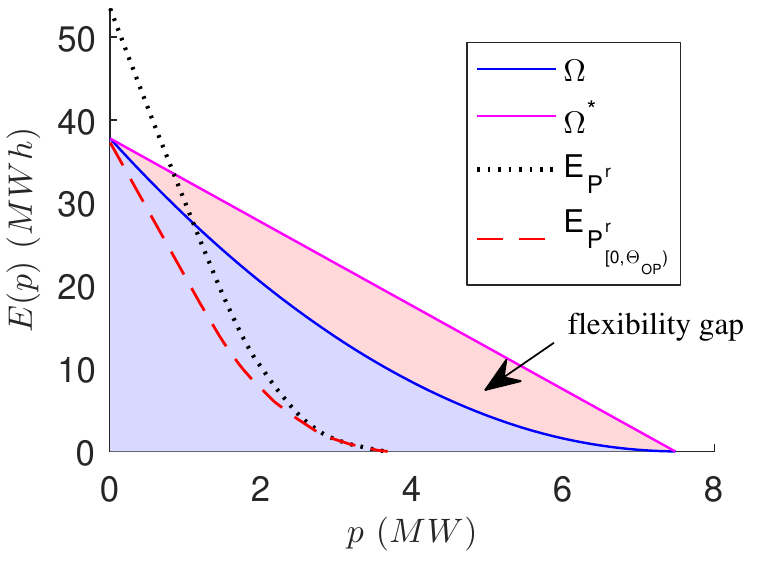}
	\caption{Use of the $E$-$p$ transform to determine reference feasibility, for the case study of Figure~{\ref{fig:AvailPlusRef}}. Capacity curves are represented by solid lines and transformed reference curves by dashed or dotted lines (consistent with Figure~{\ref{fig:AvailPlusRef}}). $\Omega$ denotes the capacity curve, and the feasible region defined by this is shaded in blue. $\Omega^*$ denotes the maximum-flexibility capacity curve; the red shaded area between these two curves is the flexibility gap.}
	\label{fig:E(p)}
\end{figure}

\subsection{Fleet capability assessment}
We here determine the capability of a fleet to provide a user-defined system service. We choose as an example service a pulse input of fixed duration, and demonstrate how the $E$-$p$ transform can be used to find the largest-magnitude feasible pulse of 15 s duration. We choose the fleet to be 10 devices with equal maximum discharge rating $\bar{p}_i\ {=}\ 1.5$~kW. We generate initial state values for half the fleet uniformly, as $x_i\ {\sim}\ U(0,20)$~s, and choose the other half to start full, at $x_i\ {=}\ 20$~s. This setup might, for example, represent uninterruptible power supplies that are allowed to take part in short-duration frequency regulation. A pulse signal will have a corresponding $E$-$p$ curve that is linear, of gradient equal to the negative of the duration. The pulse magnitude is then equal to the $p$-intercept of this curve. Our task is to find the curve of this form with the largest $p$-intercept out of those which lie in the feasible region. Figure~\ref{fig:primary/secondary} demonstrates this approach across three pulses of varying magnitude, which are then shown in Figure~\ref{fig:primary/secondary2}, and it can be seen that an 11.2~kW pulse is the largest that the fleet can feasibly meet. Note that this example is composed at a drastically smaller device scale than was previously explored, demonstrating how the presented technique is applicable regardless of the scale of the problem.

\begin{figure}[h]
	\centering
	\begin{subfigure}[l]{0.65\columnwidth}
		\centering
		\includegraphics[width=\columnwidth]{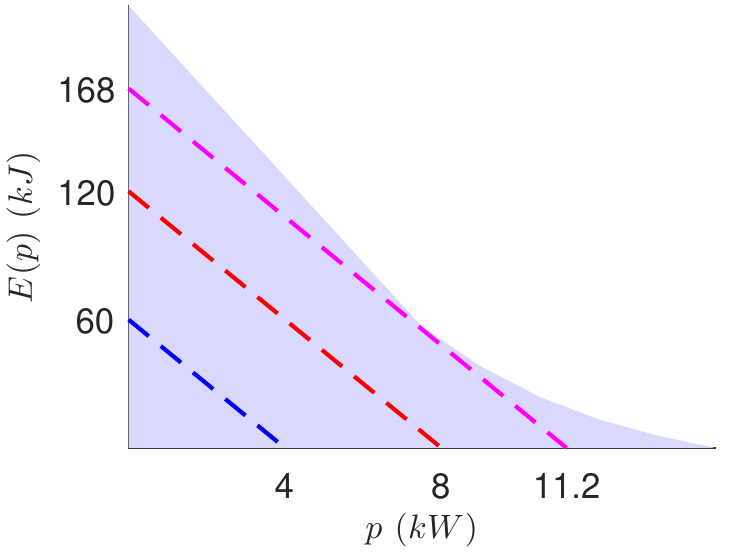}
		\caption{Three example $E$-$p$ curves corresponding to pulse inputs of 15 s duration. The magnitude of each pulse is equal to the curve's $p$-intercept. The blue shaded area represents the feasible region.}
		\label{fig:primary/secondary}
	\end{subfigure}%
	\quad
	\begin{subfigure}[l]{0.6\columnwidth}
		\centering
		\vspace{5mm}
		\includegraphics[width=\columnwidth]{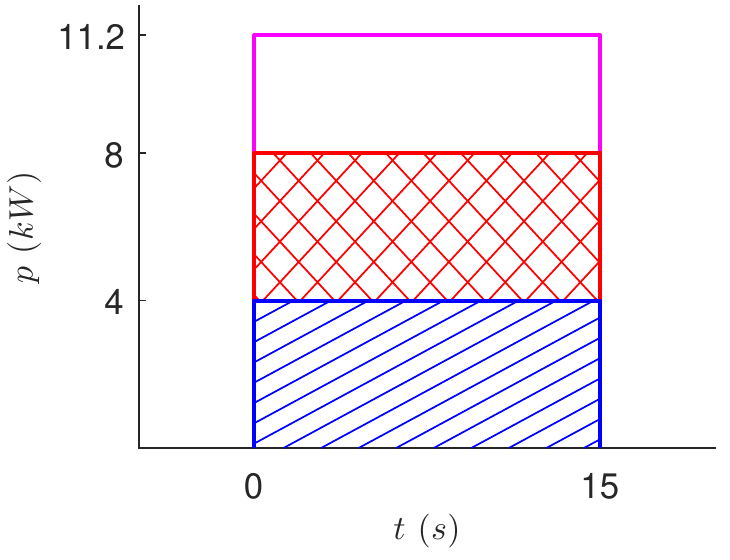}
		\caption{The three pulses of 15 s duration considered. To enable a direct comparison, the unfilled magenta pulse lies behind the cross-hatched red pulse, which in turn lies behind the hatched blue pulse.}
		\label{fig:primary/secondary2}
	\end{subfigure}
	\caption{Use of the $E$-$p$ transform to determine the largest feasible pulse magnitude of 15 s duration. The colour of each $E$-$p$ curve in Figure~\ref{fig:primary/secondary} matches that of the corresponding pulse signal in Figure~\ref{fig:primary/secondary2}.}
\end{figure}

\section{Discussion}
\label{sec:disc}
We here discuss the consequences of the above results. Firstly, as we have shown, the feasibility of a received or derived reference (for example that corresponding to a particular ancillary service) can be immediately determined. This analysis could also be extended into stochastic settings, i.e. those with random state, maximum power or reference values, by sampling scenario traces and rapidly checking the feasibility of each. This could, for example, allow a network operator to devise a day-ahead plan based upon a forecast as follows. If the probability, found through a Monte Carlo simulation, that the storage will be unable to meet the net demand exceeds some user-defined limit, the network operator could preventively take measures to reduce the high risk of loss of load. Alternatively, robust estimates could be composed from lower bounds on state or maximum power values, or from upper bounds on reference values. Similarly, the transform could be embedded in a range of algorithms, either as a rapid binary feasibility check or as inequality constraints in an optimisation problem. It is here where the computational improvements of the proposed transform over straightforward simulation might become crucial. 

Ongoing work involves extending these results into scenarios in which infeasibility is found to occur. The one-to-one mapping between feasible references and those that will be met by the optimal policy enables one to identify routes to feasibility in these cases. Moreover, the convenient mathematical properties (convexity, monotonicity) of the $E$-$p$ transform offer significant advantage in the construction of extended control strategies that minimise energy shortfalls.

We now briefly mention the scalability of the optimal policy in practical implementations. One need only broadcast the following three values: 1) the state above which devices should become active at full speed, 2) the state above which devices should become active at a fraction of full speed and 3) that fraction. The complexity of the broadcast will therefore remain unchanged as the number of devices increases. The calculation of these three values, however, requires up to one pass through the devices and will therefore scale as $\mathcal{O}(n)$. Similarly, it is worth mentioning the scalability of the proposed transform. This is formed via a single pass through the device fleet, hence scales as $\mathcal{O}(n)$. Where necessary, in order to decrease the computational complexity below this, it would also be possible to cluster devices by time-to-go and approximate from below the capacity curve for each cluster (using its smallest time-to-go value). The aggregated capacity curve could then be used to obtain a robust determination of feasibility.

\section{Conclusions and future work}
\label{sec: conclusion}
This paper has extended the results of \cite{Evans2017}, in the form of additional analytic properties of a system in which storage is used for grid support. We have conceived a transform that compactly represents the capacity of the system and can be used to characterise fleet capability; this then complements the optimal policy when making decisions at longer timescales. The transform can be used to calculate the maximum provision by the fleet of a range of system services, as well as directly compare flexibility across different fleets. In addition the $E$-$p$ transform allows for simple and immediate determination of the feasibility of any received request profile, through a comparison of $E$-$p$ curves. We have discussed the ways in which these assessments could be extended to stochastic settings. The optimal policy can of course also be used for device scheduling across a range of applications.

This work has predominantly considered settings in which references are feasible, only performing our analysis up until the time to failure under the optimal policy. In future work the authors intend to extend this into scenarios in which a reference cannot be met for all time, in which case they plan to investigate the policy that would result in the least amount of lost load. Additional future extensions include the incorporation of charging requests and cross-charging among devices. 

\section*{Appendix}
\label{sec:appendix}
\refstepcounter{section}
\addcontentsline{toc}{section}{Appendix}
\renewcommand{\theequation}{A.\arabic{equation}}	
\setcounter{equation}{0}
\subsection{Proof of Lemma~\ref{lem:commutable refs}}
\proofCommutableRefs
\subsection{Proof of Lemma~\ref{lem:terminal state}}
\proofTerminalState
\subsection{Proof of Lemma~\ref{lem:convex}}
\proofConvex
\subsection{Proof of Theorem~\ref{the:E_p}}
For this proof we initially form the following supporting Lemmas:
\begin{lemma}
	\label{lem:E-p piecewise const}
	If a \textbf{piecewise constant} reference signal $P^r(\cdot)$ is feasible, its $E$-$p$ transform is dominated by the capacity of the system, i.e. $P^r(\cdot)\in\mathcal{F}_{\bar{p},x}\implies E_{P^r}(p)\leq \Omega_{\bar{p},x}(p)\ \forall p$.
\end{lemma}
\begin{proof}
	Firstly, the result of Lemma~\ref{lem:commutable refs}, combined with the fact that $P^r(\cdot)$ is both piecewise constant and feasible, allows us to form a non-increasing equivalent to $P^r(\cdot)$, which we denote $\tilde{P}^r(\cdot)$ and take to be the reference of interest for the remainder of this proof.
	
	We then compose a framework as follows. Consider the partition induced at the initial time by the optimal policy, and let $\mathcal{N}_i,\ i=1,2,...,l$ be the corresponding subsets of devices. Making use of these subsets, then form the worst-case reference of the system, $R(\cdot)$ as in \eqref{eq:worst-case_ref}, so that this can be compared to $\tilde{P}^r(\cdot)$. We note that $R(\cdot)$ is bang-bang in all components of the input (i.e. it runs them at full power from time $t=0$ up to depletion). Denoting by $\tilde{E}(\cdot)$ and $E^*(\cdot)$ the energy vectors resulting from implementation of the optimal policy to meet references $\tilde{P}^r(\cdot)$ and $R(\cdot)$ respectively, we are therefore able to say that $\tilde{E}(t)\geq E^*(t)\ \forall t\geq0$. Moreover, since $\tilde{P}^r(\cdot)$ and $R(\cdot)$ are both feasible,
	\begin{gather}
	1^T\dot{\tilde{E}}(t)=-\tilde{P}^r(t)\ \ \forall t\geq0, \\
	1^T\dot{E}^*(t)=-R(t)\ \ \forall t\geq0.
	\end{gather}
	Now, consider any power level $p$, and define the following:
	\begin{gather}
	t_P\doteq\inf\{t\geq0\colon\tilde{P}^r(t)\leq p\}, \\
	t_R\doteq\inf\{t\geq0\colon R(t)\leq p\},
	\end{gather}
	which exist finite. We are then directly able to make the following comparison:
	\begin{equation}
	\begin{aligned}
	E_{\tilde{P}^r}(p)
	&=\int_0^{t_P}[-1^T\dot{\tilde{E}}(t)-p]dt \\
	&=1^T\tilde{E}(0)-1^T\tilde{E}(t_P) - pt_P \\
	&\leq1^TE^*(0)-1^TE^*(t_P) - pt_P \\
	&=\int_0^{t_P}[-1^T\dot{E}^*(t)-p]dt\\
	&\leq\int_0^{t_R}[R(t)-p]dt
	= \Omega_{\bar{p},x}(p),
	\end{aligned}
	\end{equation}
	where the final inequality follows from consideration of the cases $t_P< t_R$, $t_P= t_R$ and $t_P> t_R$.
\end{proof}

\begin{lemma}
	\label{lem:cont_2-ref_comp}
	Given two reference signals $P^r(\cdot)$ and $\tilde{P}^r(\cdot)$, if
	\begin{equation}
	P^r(t)\geq \tilde{P}^r(t)\geq 0 \ \ \forall t
	\label{eq:cont_2-ref_comp}
	\end{equation}
	and $P^r(\cdot)$ is feasible, then $\tilde{P}^r(\cdot)$ is also feasible.
\end{lemma}
\begin{proof}
	Since $P^r(\cdot)$ is feasible, $\exists u(\cdot)$ induced by this reference which satisfies all necessary conditions. Consider now constructing an input $\tilde{u}(\cdot)$ simply as
	\begin{equation}
	\tilde{u}(t)=\left\{
	\begin{array}{@{}ll@{}}
	\frac{\tilde{P}^r(t)}{P^r(t)}u(t), & \text{if}\ P^r(t)>0 \\
	0, & \text{otherwise}.
	\end{array}\right.
	\end{equation}
	Noting firstly that \eqref{eq:cont_2-ref_comp} leads to $P^r(t)=0\implies\tilde{P}^r(t)=0$, we are able to see by construction that $1^T\tilde{u}(t)=\tilde{P}^r(t)\ \ \forall t$. In addition the feasibility of $P^r(\cdot)$ combined with $0\leq\tilde{u}(t)\leq u(t)\ \forall t$ leads to the result that $\tilde{u}(\cdot)$ must meet all the other necessary conditions for feasibility as in Definition~\ref{def:feasible}.
\end{proof}

\begin{lemma}
	\label{lem:feasible implies E_p}
	If a \textbf{piecewise continuous} signal $P^r(\cdot)$ is feasible, its $E$-$p$ transform is dominated by the capacity of the system, i.e. $P^r(\cdot)\in\mathcal{F}_{\bar{p},x}\implies E_{P^r}(p)\leq \Omega_{\bar{p},x}(p)\ \forall p$.
\end{lemma}
\begin{proof}
	Given $P^r(\cdot)$, consider constructing $k$ piecewise constant signals, $\tilde{P}^{r,i}(\cdot),\ i=1,...,k$, the sum of which approximates from below the original signal, i.e.
	\begin{subequations}
	\begin{gather}
	P^{r,k}(t)\doteq\sum_{i=1}^k\tilde{P}^{r,i}(t)\leq P^r(t)\ \ \forall t,\\
	\lim_{k\to+\infty}\bigg\lVert\sum_{i=1}^k\tilde{P}^{r,i}-P^r(t)\bigg\rVert_\infty=0.
	\end{gather}
	\end{subequations}
	We know from the result of Lemma~\ref{lem:cont_2-ref_comp} that $P^{r,k}(\cdot)$ is feasible, and in particular that $E_{P^{r,k}}(p)\leq\Omega_{\bar{p},x}(p)\ \forall p,k$. Letting $k\to\infty$, we then see that $E_{P^r}(p)=\lim_{k\to\infty}E_{P^{r,k}}(p)\leq\Omega_{\bar{p},x}(p)\ \forall p$.
\end{proof}

\begin{lemma}
	\label{lem:E_p implies feasible}
	If the $E$-$p$ transform of a reference signal $P^r(\cdot)$ is dominated by the capacity of the system, the signal is feasible, i.e.
	\begin{equation}
	E_{P^r}(p)\leq \Omega_{\bar{p},x}(p)\ \ \forall p\implies P^r(\cdot)\in\mathcal{F}_{\bar{p},x}.
	\label{eq:E_p_other_way}
	\end{equation}
\end{lemma}
\begin{proof}
	For a reference signal to be feasible there must exist an input that meets all power and energy constraints as in Definition~\ref{def:feasible}; to prove the claim it is sufficient to find such an input signal. We apply the optimal feedback \eqref{eq: explicit policy} and construct accordingly our chosen input, which meets all power constraints by construction. All that remains is to show that it does not result in negative energy stored in any device, which we do as follows.
	
	Consider the partition induced by the optimal policy at a state $x=x(t)$. Let $\mathcal{N}_{1}(x),\mathcal{N}_{2}(x),...,\mathcal{N}_{q(x)}(x)$, in which $q(x)$ is the number of unique values in $x$, be the corresponding subsets ordered by decreasing time-to-go. Additionally denote the maximum power vectors corresponding to each subset as $\bar{U}_{1}(x),\bar{U}_{2}(x),...,\bar{U}_{q(x)}(x)$. Now, consider the power level $s$ defined according to the final state as follows:
	\begin{equation}
	s\doteq\sum_{i=1}^{q(x(\infty))-1}1^T\bar{U}_{i}(x(\infty)).
	\end{equation}
	This power level corresponds to a summation across all devices except members of the subset with (equal) smallest time-to-go value at the final time. We choose this level because the last subset of devices would be the first to be pushed negative under the optimal policy; if this subset has non-negative time-to-go then all devices must have non-negative time-to-go.
	
	Now, we know that the subsets of devices are monotonically expanding in time. Hence the final subsets are well-defined and in addition are composed of subsets formed at any previous time instant. Thus we know that the power level $s$ corresponds to a summation across one or more subsets at any time along the considered solution. We define
	\begin{equation}
	l\colon x(t)\to k\colon \sum_{i=1}^{k}1^T\bar{U}_{i}(x(t))=s.
	\end{equation}
	Consider now the worst-case reference from an arbitrary starting state $x$.	We are able to write this and the corresponding capacity explicitly as a function of the condensed state vector $X$ as follows:
	\begin{subequations}
		\begin{gather}
		R(t,X)=\sum_{j=1}^{q(x)}\bigg[H(t)-H(t-x_{\mathcal{N}_{j}})\bigg]1^T\bar{U}_{j}(x),\\
		\Omega(p,X)=\int_0^\infty\max\{R(t,X)-p,0\}dt,
		\end{gather}
	\end{subequations}
	in which $H(\cdot)$ denotes the Heaviside step function. This expression is only valid for $x$ non-negative, however we extend it to arbitrary $x$ as follows. We invert the correspondence of the worst-case reference to give time as a function of power:
	\begin{equation}
	t(p,X)=\left\{
	\begin{array}{@{}ll@{}}
	\infty, & \text{if}\ p=0 \\
	x_{\mathcal{N}_1}, & \text{if}\ 0<p\leq 1^T\bar{U}_1(x) \\
	x_{\mathcal{N}_2}, & \text{if}\ 1^T\bar{U}_1(x)< p\leq \sum_{i=1}^21^T\bar{U}_i(x) \\
	\vdots \\
	x_{\mathcal{N}_{q(x)}}, & \text{if}\ \sum_{i=1}^{q(x)-1}1^T\bar{U}_i(x)< p\leq p_{max} \\
	0, & \text{otherwise},
	\end{array}\right.
	\end{equation}
	in which $p_{max}\doteq \sum_{i=1}^{q(x)}1^T\bar{U}_i(x)=\sum_{j=1}^n\bar{p}_j$, and we then redefine the capacity as an integration along the $p$-axis:
	\begin{equation}
	\Omega(p,X)\doteq\int_p^{p_{max}}t(p',X)dp'.
	\end{equation}
	
	\noindent Given this generalised definition of the capacity, consider taking partial derivatives with respect to each distinct state, evaluated at the power level $s$ (recall that this is independent of the current state). This leads to
	\begin{equation}
	\frac{\partial\Omega}{\partial x_{\mathcal{N}_{i}}}\bigg\vert_{s}=\left\{
	\begin{array}{@{}ll@{}}
	1^T\bar{U}_{i}(x), & \text{if}\ i>l(x) \\
	0, & \text{otherwise}.
	\end{array}\right.
	\end{equation}
	Now, denoting by $m(x)$ the highest group index of devices running at positive power, we know that the dynamics corresponding to the optimal policy \eqref{eq: explicit policy} are
	\begin{subequations}
	\begin{equation}
	\dot{x}_{\mathcal{N}_{i}}=\left\{
	\begin{array}{@{}ll@{}}
	-1, & \text{if}\ i<m(x) \\
	-r_i, & \text{if}\ i=m(x) \\
	0, & \text{if}\ i>m(x),
	\end{array}\right.
	\end{equation}	
	in which
	\begin{equation}
	r_i=\frac{P^r-\sum_{j=1}^{i-1}1^T\bar{U}_{j}(x)}{1^T\bar{U}_{i}(x)}.
	\label{eq:frac_of_full}
	\end{equation}
	\end{subequations}
	We are then able to consider the 6 possible cases that arise for the following product for each subset $\mathcal{N}_{i}$ and at the power level $s$:
	\begin{equation}
	\dot{x}_{\mathcal{N}_{i}}\cdot\frac{\partial\Omega}{\partial x_{\mathcal{N}_{i}}}\bigg\vert_{s},
	\label{eq:partial_diff_prod}
	\end{equation}
	as seen in Table~\ref{tab:6cases}.
	\renewcommand{\arraystretch}{1.3}
	\begin{table}[h]
		\centering
		\begin{tabular}{|r@{}|c@{}c|}
				\hline
				& $i\leq l(x)$ & $i>l(x)$ \\
				\hline
				$i<m(x)$ & 0 & $-1^T\bar{U}_{i}(x)$ \\
				$i=m(x)$ & 0 & $-r_i1^T\bar{U}_{i}(x)$ \\
				$i>m(x)$  & 0 & 0 \\
				\hline
		\end{tabular}
		\caption{The 6 possible cases for the expression of \eqref{eq:partial_diff_prod}.}
		\label{tab:6cases}
	\end{table}
	\renewcommand{\arraystretch}{1}
	
	\noindent As it can be seen from the table, there are only 2 cases in which the expression of \eqref{eq:partial_diff_prod} is non-zero: when $l(x)<i\leq m(x)$. These can be converted back to power conditions as follows:
	\begin{equation}
	\begin{aligned}
	&\dot{x}_{\mathcal{N}_{i}}\cdot\frac{\partial\Omega}{\partial x_{\mathcal{N}_{i}}}\bigg\vert_{s}\\
	&\hspace{5mm}=\left\{
	\begin{array}{@{}ll@{}}
	-1^T\bar{U}_i(x), & \text{if}\ s<\sum_{j=1}^{i}1^T\bar{U}_j(x)\leq P^r \\
	\sum_{j=1}^{i-1}1^T\bar{U}_j(x)-P^r, & \text{if}\ s< P^r<\sum_{j=1}^{i}1^T\bar{U}_j(x) \\
	0, & \text{otherwise}.
	\end{array}\right.
	\end{aligned}
	\end{equation}
	Utilising this expression, we can then use the chain rule to form the partial derivative with respect to time of the capacity, evaluated at power level $s$ and state $x$, by summing across subsets. This operation is valid because each subset of devices can in effect be aggregated into a single virtual device, and leads to
	\begin{equation}
	\begin{aligned}
	&\frac{\partial\Omega}{\partial t}\bigg\vert_{s,x} =
	\sum_{i=1}^{q(x)}\dot{x}_{\mathcal{N}_{i}}\cdot\frac{\partial\Omega}{\partial x_{\mathcal{N}_{i}}}\bigg\vert_{s} \\
	&\hspace{2mm} \overset{(\dagger)}{=}-\sum_{i=l(x)+1}^{q(x)}\max\bigg\{0,\min\bigg\{P^r-\sum_{j=1}^{i-1}1^T\bar{U}_j(x),	1^T\bar{U}_i(x)\bigg\}\bigg\} \\
	&\hspace{2mm}=\sum_{i=1}^{l(x)}\max\bigg\{0,\min\bigg\{P^r-\sum_{j=1}^{i-1}1^T\bar{U}_j(x),	1^T\bar{U}_i(x)\bigg\}\bigg\}\\
	&\hspace{2mm}\hphantom{=}\ -\sum_{i=1}^{q(x)}\max\bigg\{0,\min\bigg\{P^r-\sum_{j=1}^{i-1}1^T\bar{U}_j(x),	1^T\bar{U}_i(x)\bigg\}\bigg\} \\
	&\hspace{2mm}\overset{(\ddagger)}{=} \min\big\{P^r,s\big\}-P^r
	=-\max\big\{P^r-s,0\big\},
	\end{aligned}
	\label{eq:dOmega/dt}
	\end{equation}	
	in which the equality $(\dagger)$ results from the removal of the cases for which $l(x)\geq i$ leading to a product of 0, and the equality $(\ddagger)$ results from the following arguments. The negative summation can be interpreted as the negative of the aggregate power across all devices, which must equal $P^r$ as a result of the construction of the input according to the optimal policy. The positive summation is the same series curtailed at the index corresponding to the level $s$; if $P^r>s$ then all devices up to this level will be run at full power and their aggregate will equal $s$, otherwise the aggregate will once again be the total power $P^r$. Integration of \eqref{eq:dOmega/dt} gives
	\begin{equation}
	\Omega\big(s,x(\infty)\big)=\Omega\big(s,x(0)\big)-\int_0^\infty\max\big\{P^r(t)-s,0\big\}dt,
	\end{equation}
	which coupled with the condition \eqref{eq:E_p_other_way} leads directly to $\Omega\big(s,x(\infty)\big)\geq0$.	Moreover,
	\begin{equation}
	\Omega\big(s,x(\infty)\big)=x_{\mathcal{N}_{q(x(\infty))}}(\infty)\cdot 1^T\bar{U}_{q(x(\infty))}\big(x(\infty)\big),
	\end{equation}
	therefore $x_{\mathcal{N}_{q(x(\infty))}}(\infty)\geq0$. As the time-to-go values cannot increase over time the reference must therefore be feasible.
\end{proof}
\noindent Theorem~\ref{the:E_p} then follows directly from Lemmas~\ref{lem:feasible implies E_p} and \ref{lem:E_p implies feasible}.
\qed

\bibliographystyle{IEEEtran}
\bibliography{bib_new2,biblio_manual2}
\end{document}